\documentclass[11pt]{amsart}
\usepackage[utf8]{inputenc}

\usepackage{amsmath}
\usepackage{amssymb}
\usepackage{enumitem}
\usepackage{amsthm}
\usepackage{graphicx}
\usepackage{mathabx} 
\usepackage{a4wide}

\numberwithin{equation}{section}

\newtheorem{thmx}{Theorem}

\newtheorem{theorem}{Theorem}[section]
\newtheorem{proposition}[theorem]{Proposition}
\newtheorem{lemma}[theorem]{Lemma}
\newtheorem{corollary}[theorem]{Corollary}
\theoremstyle{definition}
\newtheorem{definition}[theorem]{Definition}
\theoremstyle{remark}
\newtheorem*{remark}{Remark}
\newtheorem{example}[theorem]{Example}

\DeclareMathOperator{\Irr}{Irr}

\newcommand{\B}{\mathcal{B}}
\newcommand{\C}{\mathbb{C}}

\newcommand{\N}{\mathbb{N}}
\newcommand{\QG}{\mathbb{G}}
\newcommand{\QH}{\mathbb{H}}
\newcommand{\QK}{\mathbb{K}}
\newcommand{\Pol}{\mathcal{O}}
\newcommand{\I}{\mathcal{I}}
\newcommand{\R}{\mathbb{R}}
\newcommand{\T}{\mathbb{T}}
\newcommand{\Z}{\mathbb{Z}}

\DeclareMathOperator{\Gauss}{Gauss}
\DeclareMathOperator{\id}{id}

\begin{document}

\title{Connectedness and Gaussian Parts for Compact Quantum Groups}

\author{Uwe Franz}
\address{U.F., D\'epartement de math\'ematiques de Besan\c{c}on,
Universit\'e de Bourgogne Franche-Comt\'e, 16, route de Gray, 25 030
Besan\c{c}on cedex, France}
\email{uwe.franz@univ-fcomte.fr}
\urladdr{http://lmb.univ-fcomte.fr/uwe-franz}

\author{Amaury Freslon}
\address{A.F., Universit\'e Paris-Saclay, CNRS UMR 8628, Laboratoire de Mathématique d'Orsay, 91405 Orsay cedex, France}
\email{amaury.freslon@universite-paris-saclay.fr}
\urladdr{https://www.imo.universite-paris-saclay.fr/~freslon}

\author{Adam Skalski}
\address{A.S. Institute of Mathematics of the Polish Academy of Sciences, ul.\ \'Sniadeckich 8, 00-656 Warszawa, Poland}
\email{a.skalski@impan.pl}

\thanks{U.F.\ and A.F.\ were partially supported  by the ANR grant ``Noncommutative analysis on groups and quantum groups'' (ANR-19-CE40-0002).
A.F.\  was also partially supported by the ANR grant ``Operator algebras and dynamics on groups'' (ANR-19-CE40-0008).
A.S.\ was partially supported by the National Science Center (NCN) grant no.~2020/39/I/ST1/01566. 
\\ \hspace*{10pt} We thank Jacek Krajczok for useful comments.
}

\begin{abstract}
We introduce the Gaussian part of a compact quantum group $\QG$, namely the largest quantum subgroup of $\QG$ supporting all the Gaussian functionals of $\QG$. We prove that the Gaussian part is always contained in the Kac part, and characterise Gaussian parts of classical compact groups, duals of classical discrete groups and $q$-deformations of compact Lie groups. The notion turns out to be related to a new concept of ``strong connectedness'' and we exhibit several examples of both strongly connected and  totally strongly disconnected compact quantum groups.
\end{abstract}
\maketitle

\section{Introduction}
\label{sec:intro}

L\'evy processes, i.e.\ stochastic processes with independent and identically distributed increments, form one of the most studied classes of stochastic processes (\cite{Sato}); and among these the Gaussian processes, and in particular the Brownian motion are crucial examples. They were initially studied in the Euclidean space, but it quickly became clear that  the convolution product for probability measures afforded by the (locally compact) group structure gives a very natural framework for generalisations (\cite{Heyer}, \cite{Liao}). In the end of 1980s, motivated by the development of quantum probability theory, the algebraic language was used by Accardi, Sch\"urmann and von Waldenfels to define in \cite{ASW} abstract \emph{quantum L\'evy processes}, understood as certain families of quantum random variables living on a $^*$-bialgebra. Soon after, Sch\"urmann introduced in \cite{sch90} an important subclass of quantum L\'evy processes, namely \emph{quantum Gaussian processes}, see also \cite[Section 5]{schurmann93}. Sch\"urmann's definition is phrased in terms of \emph{Gaussian generating functionals}. These should be viewed as quantum counterparts of `second-order' or `quadratic' generators of convolution semigroups of measures, and naturally lead to the concept of \emph{Gaussian states on (compact) quantum groups}. The idea of Sch\"urmann opened a path to studying quantum versions of L\'evy-Khintchine decompositions (\cite{das+al18} and references therein) or to classification of Gaussian generators on concrete quantum groups (see \cite{SchurmannSkeide}). 

The last cited paper exhibited a curious phenomenon: all Gaussian generating functionals on Woronowicz's $SU_q(2)$ (with $q \in (-1,0)\cup (0,1)$) are supported on the circle $\mathbb{T}$, the largest classical subgroup of $SU_q(2)$. In the same spirit, it is known that the free permutation do not support any non-trivial Gaussian generating functionals whatsoever (\cite{FranzKulaSkalski}). Meanwhile, some examples show that in general Gaussian generating functionals need not be supported on classical subgroups (see \cite{das+al18} or the results in Sections \ref{sec:ExamplesI} and \ref{sec:ExamplesII} below). This motivates the main question studied in this paper, which has natural connections for example with the notion of topological generation for compact quantum groups (see the end of Section \ref{subsec:free}):
\begin{itemize}
\item Given a compact quantum group $\QG$, what is the smallest quantum subgroup of $\QG$ which supports all Gaussian generating functionals (equivalently, all Gaussian states) of $\QG$? 	
\end{itemize}
We will call the quantum subgroup as above (whose existence is easy to deduce) the \emph{Gaussian part} of $\QG$ and denote it $\Gauss(\QG)$, by analogy with the \emph{classical part}  (i.e.\ the largest classical subgroup) of $\QG$ and the \emph{Kac part} (the largest quantum subgroup of Kac type) of $\QG$. The analysis of the classical case shows that there the notion of the Gaussian part coincides with the connected component of the identity. This motivates the introduction of a natural quantum counterpart of classical connectedness. To distinguish this notion from a strictly weaker concept defined in \cite{wang2009simple}, we call it \emph{strong connectedness}; roughly speaking a compact quantum group $\QG$ is strongly connected if the intersection of a certain family of  ideals in $\Pol(\QG)$  trivialises. With this notion at hand -- which appears to be of independent interest -- we can state the first main result of this work.

\begin{thmx}\label{Theorem_A}
For any compact quantum group $\QG$ the Gaussian part of $\QG$ is contained in the strongly connected component of identity of $\QG$, which in turn is contained in the Kac part of $\QG$:
\begin{equation*}
\Gauss(\QG) \subset \QG^{00} \subset \textup{Kac}(\QG).
\end{equation*}
\end{thmx}

Apart from the structural result above we compute the Gaussian parts of many compact quantum groups. We summarise the main statements obtained in the following theorem. 

\begin{thmx}\label{Theorem_B}
The following hold:	
\begin{itemize}
\item If $G$ is a classical compact group, then its Gaussian part coincides with its (strongly) connected component	of the identity: $\Gauss(G) = G^0$;
\item If $\Gamma$ is a finitely generated  discrete group, then $\Gauss(\widehat{\Gamma}) = \widehat{\Gamma/\sqrt{\gamma_{3}(\Gamma)}}$;
\item Gaussian parts of (pro-)finite quantum groups and of quantum permutation groups are trivial;
\item If $G$ is a simply connected semisimple compact Lie group and $q\in (0,1)$ then $\Gauss(\QG_q)$ is the maximal torus $\mathbf{T}\subset  {\QG_q}$;
\item $\Gauss(O_{N}^{*})=SO(N)$;
\item $\Gauss(O_{N}^{+})$ (for $N \geq 4$) and  $\Gauss(U_{N}^{+})$ (for $N \geq 2$)  are neither classical nor dual to  discrete groups. 
\end{itemize}
\end{thmx}
It is worth noting that we do not know whether the Gaussian part of the free orthogonal group  $O_N^+$ is $O_N^+$ itself (similarly for the free unitary group). The problem seems to have subtle connections with the questions regarding topological generation inside free quantum groups.

The detailed plan of the paper is as follows: in Section \ref{sec:Gaussian} we recall the basic facts concerning generating functionals on $^*$-bialgebras, including the notion of Gaussian generating functionals and the subclass of drifts, together with several related characterizations, and establish a Wick-type formula which is useful for computations.
Section \ref{sec:Gaussianpart} recalls basic facts of the theory of compact quantum groups, introduces the Gaussian part and establishes its basic properties. We also prove a that point that Gaussian parts satisfy the Kac property, and the proof relies on a new characterization of the maximal Kac quotient which we believe to be of independent interest. Section \ref{sec:stronglyconnected} is devoted to the concept of strong connectedness; there we also complete the proof of Theorem A. Eventually, the last two sections contain computations of the Gaussian part in several examples. In Section \ref{sec:ExamplesI} we determine Gaussian (and drift) parts for classical compact groups and duals of discrete groups, and in Section \ref{sec:ExamplesII} we do the same for $q$-deformations and half-liberated orthogonal groups and discuss partial results we obtained for free quantum groups. This completes the proof of Theorem B, which is a combination of Corollary \ref{cor:GaussClassical}, Theorem \ref{dual:Gaussian}, Proposition \ref{prop:projections}, Proposition \ref{prop:q-deform}, Proposition \ref{prop:halfliberated} and Proposition \ref{prop:freeqg}.


\section{Gaussian functionals}\label{sec:Gaussian}

In this section we recall the notion of generating functionals on $^*$-bialgebras, focusing on the class of Gaussian functionals and drifts for which we provide several characterizations. We also show that Gaussian generating functionals satisfy a version of the Wick property.

\subsection{Definition and properties}

Let $\B$ be an involutive bialgebra with unit $\mathbf{1}$ and counit $\varepsilon$. In all cases of interest afterwards, $\B$ will be the Hopf *-algebra $\Pol(\QG)$ of a compact quantum group $\QG$, but we stay at a general level for the moment. A \emph{generating functional} on $\B$ is a linear functional $\phi : \B\to \C$ with the following three properties:
\begin{enumerate}
\item 
$\phi(\mathbf{1}) = 0$ (normalization);
\item 
$\phi(b^{*}) = \overline{\phi(b)}$ for all $b\in \B$ (hermitianity);
\item
$\phi(b^{*}b)\geqslant 0$ for all $b\in {\rm ker}\,\varepsilon$ (conditional positivity).
\end{enumerate}

We are interested in generating functionals, because it follows from \cite[Section 3.2]{schurmann93} that they are in one-to-one correspondence with convolution semigroups of states.

\begin{proposition}
Let $\B$ be an involutive bialgebra and let $\phi : \B\to \C$ a linear functional. Set, for $t\in \R_{+}$,
\begin{equation*}
\varphi_{t} = \exp_{\star}(t\phi) := \sum_{n=0}^{+\infty} \frac{(t\phi)^{\star n}}{n!},
\end{equation*}
with the convention $\phi^{\star 0} = \varepsilon$. Then, the following are equivalent:
\begin{enumerate}
\item 
$\phi$ is a generating functional;
\item 
$\varphi_{t}$ is a \emph{state} for all $t\geqslant 0$, in the sense that
\begin{enumerate}
\item 
$\varphi_t(\mathbf{1})=1$;
\item 
$\varphi_t(b^{*}b)\geqslant 0$ for all $b\in \B$. 
\end{enumerate}
\end{enumerate}
\end{proposition}

By a GNS-type construction one can associate a so-called \emph{Sch\"urmann triple} to a generating functional, cf.\ \cite[Subsection 2.3]{schurmann93}. Let us recall how this works. For an inner product space $D$ (pre-Hilbert space), we denote by
\[
\mathcal{L}(D) = \big\{ X:D\to D\text{ linear} \mid \exists X^* : D\to D \text{ s.t.\ }\forall u,v\in D,  \langle u,Xv\rangle = \langle X^*u,v\rangle \big\}
\]
the *-algebra of adjointable linear maps on $D$.

\begin{definition}
Let $(\B, \varepsilon)$ be a pair consisting of a *-algebra $\B$ and a *-homorphism $\varepsilon : \B\to \C$ (an augmented *-algebra), and let $D$ be a pre-Hilbert space. A Sch\"urmann triple on $(\B, \varepsilon)$ over $D$ is a family of three linear maps $(\rho : \B\to \mathcal{L}(D), \eta : \B\to D, \phi : \B\to \C)$ such that
\begin{enumerate}
\item 
$\rho$ is a unital *-homomorphism;
\item 
$\eta$ satisfies
\begin{equation}\label{eq-cocycle}
\eta(ab) = \rho(a)\eta(b) + \eta(a)\varepsilon(b)
\end{equation}
for all $a, b\in \B$,
\item 
$\phi$ is hermitian and satisfies
\begin{equation} \label{eq-cobound}
\phi(a^{*}b)-\overline{\varepsilon(a)}\phi(b) - \overline{\phi(a)}\varepsilon(b) = \langle\eta(a),\eta(b)\rangle
\end{equation}
for all $a ,b\in \B$.
\end{enumerate}
\end{definition}

Relations \eqref{eq-cocycle} and \eqref{eq-cobound} imply $\eta(\mathbf{1}) = 0$ and $\phi(\mathbf{1})=0$. Relation \eqref{eq-cobound} furthermore shows that $\phi$ is positive on the kernel of $\varepsilon$, hence $\phi$ is a generating functional.

We call two Sch\"urmann triples $(\rho, \eta, \phi)$ and $(\rho', \eta', \phi')$ on the same augmented *-algebra $(\B, \varepsilon)$ and over pre-Hilbert spaces $D$ and $D'$ \emph{equivalent}, if there exists a surjective isometry $V : D\to D'$ s.t.
\begin{equation*}
V\eta(b) = \eta'(b) \quad\text{ and }\quad V\rho(b) = \rho'(b)V
\end{equation*}
for all $b\in \B$. Note that we thus get one-to-one correspondences between convolution semigroups of states, generating functionals, and Sch\"urmann triples with surjective cocycle (up to equivalence) on a given involutive bialgebra $\B$. Sch\"urmann \cite{schurmann93} proved that these three families of objects are also in one-to-one correspondence with \emph{L\'evy processes} on $\B$ (up to stochastic equivalence).

We now introduce a family of ideals which will play a crucial rôle in this work. We set $K_{1}(\B) = \ker(\varepsilon)$ and $K_{n}(\B) = K_{1}(\B)^{n}$ for $n\geqslant 1$, or more explicitly
\begin{equation*}
K_{n}(\B) = \mathrm{Span}\{b_{1}\cdots b_{n} \mid b_{1}, \cdots, b_{n}\in \ker(\varepsilon)\}
\end{equation*}
We also set
\begin{equation*}
K_{\infty}(\B) = \bigcap_{n\geqslant 1} K_{n}(\B).
\end{equation*}
The family $(K_{n}(\B))_{n=1}^{+\infty}$ is decreasing, and the containments can be proper, as we will see later on. In order to lighten notations, we will simply write $K_{n}$ as soon as there is no ambiguity concerning the algebra $\B$.

If $\B = \Pol(\QG)$, then $K_{1}$ is also a coideal and therefore a Hopf *-ideal, as it is the kernel of a Hopf *-algebra homomorphism. For $n\geqslant 2$ however, $K_{n}$ is in general not a coideal but nevertheless defines a filtration since
\begin{equation*}
\Delta(K_{n}) \subseteq \sum_{\ell=0}^{n} K_{\ell}\otimes K_{n-\ell} \subseteq K_{\lfloor \frac{n}{2} \rfloor} \otimes \B + \B \otimes K_{\lfloor \frac{n}{2} \rfloor}
\end{equation*}
(where $K_{0} = \B$ by convention). It turns out that if $\B$ is a Hopf algebra, then $K_{\infty}$ is a Hopf ideal, see Proposition \ref{prop:hopfideal}.

\begin{example}
Consider the *-algebra $\B = \C[x]$ of polynomials in one self-adjoint variable (i.e.\ $x^{*} = x$) and the augmentation map determined by $\varepsilon(x) = 0$. Then the ideal $K_{n}$ consists exactly of the polynomials that have a zero of order at least $n$ at the origin, and $K_{\infty} = \{0\}$.
\end{example}

We are now ready for the definition of Gaussian processes, which are the main subject of this work.

\begin{definition}
A generating functional $\phi : \B\to \C$ on an augmented *-algebra $\B$ is called \emph{Gaussian} (or quadratic, \cite[Section 5.1]{schurmann93}), if $\phi_{\mid K_{3}} = 0$.

A state on $\B$ is called Gaussian, if it is of the form $\varphi = \exp_{\star}(t\phi)$ for some $t\geqslant 0$ and $\phi$ a Gaussian generating functional. A cocycle $\eta : \B\to H$ is called Gaussian, if it is a derivation in the sense that $\eta(ab) = \varepsilon(a)\eta(b)+\eta(a)\varepsilon(b)$ for all $a, b\in \B$.
\end{definition}

The connection between the last definition and the first two ones is not obvious and relies on the following result.

\begin{proposition}\cite[Proposition 5.1.1]{schurmann93} \label{prop-gauss}
Let $\B$ be an augmented *-algebra and let $(\rho, \eta, \phi)$ be a Sch\"urmann triple on $\B$ over some pre-Hilbert space $H$ with surjective cocycle $\eta$. Then, the following are equivalent:
\begin{enumerate}
\item 
$\phi_{\mid K_{3}} = 0$,
\item 
$\phi(a^{*}a) = 0$ for all $a\in K_{2}$,
\item 
$\rho_{\mid K_{1}} = 0$,
\item 
$\rho(b) = \varepsilon(b)\mathrm{id}_{H}$ for all $b\in\B$,
\item 
$\eta_{\mid K_{2}} = 0$,
\item 
$\eta(ab) = \varepsilon(a)\eta(b) + \eta(a)\varepsilon(b)$ for all $a, b\in\B$.
\end{enumerate}
\end{proposition}
Note that the first property translates into the following condition, valid for all $a, b, c \in \B$:
\begin{equation*}
\phi(abc) = \phi(ab) \varepsilon(c) + \phi(ac) \varepsilon(b) + \phi(bc) \varepsilon(a) - \phi(a) \varepsilon(bc) - \phi(b) \varepsilon(ac) - \phi(c) \varepsilon(ab).
\end{equation*}
This gives an inductive algorithm to compute $\phi$ which leads to a Wick-type formula, see Subsection \ref{subsec:wick}.

\begin{remark}\label{remclass}
Assume that $\B$ is an augmented *-algebra, that $X \subset \B$ generates $\B$ as a *-algebra, and let $\phi:\B\to \C$ be a Gaussian generating functional with associated cocycle $\eta: \B\to D$. Then, the following conditions are equivalent:
\begin{itemize}
\item[(i)] 
$\phi(xy) = \phi(yx)$ for all $x, y \in X$; 
\item[(ii)] 
$\langle \eta(x^{*}), \eta (y)\rangle = \langle \eta(y^{*}), \eta (x)\rangle$ for all $x, y \in X$;
\item[(iii)] 
$\phi$ is a trace: $\phi(ab) = \phi(ba)$ for all $a, b \in \B$;
\item[(iv)] 
$\phi$ factors through the commutator ideal of $\B$.
\end{itemize}
Indeed, (i) and (ii) are equivalent by \eqref{eq-cobound}. Then, it suffices to observe that if (i) holds then we can prove first (iii) and then (iv) using the formula displayed before the remark; the other implications are trivial. Even though we will not need nor use the terminology hereafter, such functionals may be called \emph{classical} since they factor through a commutative algebra.
\end{remark}

Looking at the definition of a Gaussian generating functional, one may wonder at the definition obtained by strengthening the condition to $\phi_{\mid K_{2}} = 0$. A functional satisfying this condition called a \emph{drift} (or \emph{degenerate quadratic}, \cite[Section 5.1]{schurmann93}). Drifts can be characterized through their Sch\"urmann triples similarly to Gaussian functionals. 
\begin{proposition}
Under the same assumptions as in Proposition \ref{prop-gauss}, the following are equivalent:
\begin{enumerate}
\item 
$\phi_{\mid K_{2}} = 0$;
\item 
$\phi$ is a hermitian derivation, i.e.\ $\phi(a^{*}) = \overline{\phi(a)}$ and
\begin{equation*}
\phi(ab) = \varepsilon(a)\phi(b)+\phi(a)\varepsilon(b)
\end{equation*}
for all $a, b \in\B$;
\item 
$\phi(a^{*}a) = 0$ for all $a\in K_{1}$.
\end{enumerate}
\end{proposition}

\begin{remark} \label{rem:drift}
A generating functional $\phi$ is a drift if and only if $\varphi_{t} = \exp_{\star}(t\phi)$ is a character for all $t\in \R_{+}$. One way is easy: if $\varphi_{t}$ is a character for all $t\geqslant 0$, then for all $a, b\in \B$ we have 
\begin{equation*}
\phi(ab) = \lim_{t\to 0}\frac{\varphi_{t}(ab) - \varepsilon(ab)}{t} = \lim_{t\to 0}\frac{\varphi_{t}(a)\varphi_{t}(b) - \varepsilon(a)\varepsilon(b)}{t} = \phi(a)\varepsilon(b) + \phi(b)\varepsilon(a)
\end{equation*}
and combining this with the hermitian property which comes from the hermitian property of $\varphi_{t}$, we conclude by the second point of the above proposition. Conversely, if $\phi$ is a drift then a straightforward induction shows that for all $a, b \in \B$ and all $k\geqslant 0$,
\begin{equation*}
\phi^{\ast k}(ab) = \sum_{p=0}^{k}\binom{k}{p}\phi^{\ast p}(a)\phi^{\ast (k-p)}(b)
\end{equation*}
with $\phi^{\ast 0} = \varepsilon$. Multiplicativity of $\varphi_{t}$ then follows from the equalities
\begin{align*}
\varphi_{t}(a)\varphi_{t}(b)&\ = \sum_{k, k' = 0}^{+\infty}\frac{t^{k+k'}}{k!k'!}\phi^{\ast k}(a)\phi^{\ast k'}(b) \\
& = \sum_{r=0}^{+\infty}\frac{t^{r}}{r!}\sum_{i=0}^{r}\frac{r!}{i!(r-i)!}\phi^{\ast i}(a)\phi^{\ast (r-i)}(b) \\
& = \sum_{r=0}^{+\infty}\frac{t^{r}}{r!}\phi^{\ast r}(ab) = \varphi_{t}(ab).
\end{align*}
\end{remark}

\subsection{A Wick-type formula for Gaussian generating functionals}\label{subsec:wick}

The defining property of Gaussian generating functionals gives a way to compute their value on a product of elements by centering them and then applying the property recursively. This can be turned into a closed formula which is reminiscent of the Wick formula for operators on the full Fock space.

\begin{proposition}\label{prop:wick}
Let $\phi$ be a Gaussian generating functional on $\B$. Then we have for any $n\geqslant 2$ and any $a_{1}, \cdots, a_{n}\in \B$,
\begin{align*}
\phi(a_{1}\cdots a_{n}) & = \sum_{1\leqslant j< k\leqslant n} \phi(a_{j}a_{k})\varepsilon\left(a_{1}\cdots \widecheck{a_{j}}\cdots \widecheck{a_{k}}\cdots a_{n}\right) \\
& - (n-2)\sum_{1\leqslant j\leqslant n} \phi(a_{j})\varepsilon(a_{1}\cdots \widecheck{a_{j}}\cdots a_{n}),
\end{align*}
where $\widecheck{a_{j}}$ means that this factor is omitted from the product.
\end{proposition}

\begin{proof}
For $n = 2$ this is trivially true, and for $n = 3$ it is one of the equivalent characterisations in Proposition \ref{prop-gauss}.

The general case follows by induction. Let us first observe that for any $j,n \in \mathbb{N}$,
\begin{align*}
\phi(a_{j}a_{n}a_{n+1}) & = \phi(a_{j}a_{n})\varepsilon(a_{n+1}) + \phi(a_{j}a_{n+1})\varepsilon(a_{n}) + \phi(a_{n}a_{n+1})\varepsilon(a_{j}) \\
& - \phi(a_{j})\varepsilon(a_{n}a_{n+1}) - \phi(a_{n})\varepsilon(a_{j}a_{n+1}) - \phi(a_{n+1})\varepsilon(a_{j}a_{n}).
\end{align*}
Now, let $n\geqslant 3$ and $a_{1}, \cdots, a_{n+1}\in \B$. Then we have
\begin{align*}
\phi\big(a_{1}\cdots (a_{n}a_{n+1})\big) & = \sum_{1\leqslant j < k < n} \phi(a_{j}a_{k})\varepsilon\left(a_{1}\cdots \widecheck{a_{j}}\cdots \widecheck{a_{k}}\cdots a_{n}a_{n+1}\right) \\
& + \sum_{1\leqslant j < n}\underbrace{\phi(a_{j}a_{n}a_{n+1})}\varepsilon\left(a_{1}\cdots\widecheck{a_{j}}\cdots a_{n-1}\right) \\
& - (n-2)\sum_{1\leqslant j < n} \phi(a_{j})\varepsilon(a_{1}\cdots \widecheck{a_{j}}\cdots a_{n}a_{n+1}) \\
& - (n-2)\phi(a_{n}a_{n+1})\varepsilon(a_{1}\cdots a_{n-1}) \\
& = \sum_{1\leqslant j < k \leqslant n+1} \phi(a_{j}a_{k})\varepsilon\left(a_{1}\cdots \widecheck{a_{j}}\cdots \widecheck{a_{k}}\cdots a_{n+1}\right) \\
& - (n-1)\sum_{1\leqslant j\leqslant n+1} \phi(a_{j})\varepsilon(a_{1}\cdots \widecheck{a_{j}}\cdots a_{n+1}).
\end{align*}
\end{proof}

This implies immediately the following lemma which will be useful to determine conditions guaranteeing that Gaussian generating functionals vanish on certain ideals.

\begin{corollary}\label{lem-ideal}\label{cor:wickformula}
Assume that we have two subsets $X = \{a_{1}, \cdots, a_{n}\}\subseteq \B$ and $Y = \{b_{1},\cdots, b_{m}\}\subseteq \ker(\varepsilon)$ such that
\begin{enumerate}
\item $X$ generates $\B$ as an algebra;
\item $0 = \phi(b_{k}) = \phi(a_{j}b_{k}) = \phi(b_{k}a_{j})$ for all $j\in\{1, \cdots, n\}$ and all $k\in\{1, \cdots, m\}$.
\end{enumerate}
Then $\phi$ vanishes on the ideal generated by $Y$.
\end{corollary}

\begin{proof}
It suffices to show that $\phi(a_{j_{1}}\cdots a_{j_{s}}b_{k}a_{j_{s+1}}\cdots a_{j_{s+t}})$ vanishes for all $s, t\in \mathbb{N} \cup\{0\}$, $j_{1}, \cdots, j_{s+t}\in\{1, \cdots, n\}$, $k\in\{1, \cdots, m\}$. Proposition \ref{prop:wick} allows to do that by reducing this value to a linear combination of terms of the form appearing in Condition (2).
\end{proof}

\section{The Gaussian part of a compact quantum group} \label{sec:Gaussianpart}

In this section we restrict the context of our study to Hopf $^*$-algebras related to compact quantum groups. We introduce the notion of the Gaussian part of a compact quantum group, discuss its basic properties, and prove that it is neccessarily a quantum group of Kac type.

\subsection{Compact quantum groups}

In this work, we are interested in Gaussian processes on compact quantum groups. We will therefore briefly introduce these objects. We refer the reader to \cite{timmermann2008invitation} and \cite{nt13} for detailed treatments of the theory. It is known since the work of M. Dijkhuizen and T. Koornwinder \cite{dijkhuizen1994CQG} that compact quantum groups can be treated algebraically through the following notion of a CQG-algebra.

\begin{definition}
A \emph{CQG-algebra} is a Hopf $*$-algebra which is spanned by the coefficients of its finite-dimensional unitary corepresentations.
\end{definition}

If $G$ is a compact group, then its algebra of regular functions $\Pol(G)$ is a CQG-algebra. Based on that example, and in an attempt to retain the intuition coming from the classical setting, we will denote a general CQG-algebra by $\Pol(\QG)$ and say that it corresponds to the \emph{compact quantum group} $\QG$. If $\Gamma$ is a discrete group and $\C[\Gamma]$ denotes its group algebra, it is easy to endow it with a Hopf $*$-algebra structure with coproduct given by $\Delta(g) = g\otimes g$ for all $g\in \Gamma$. Since this turns each $g\in \Gamma\subset\C[\Gamma]$ into a one-dimensional co-representation, it yields a CQG-algebra. The resulting compact quantum group is called the \emph{dual} of $\Gamma$ and is denoted by $\widehat{\Gamma}$.

We will at some point use arguments involving representation theory of compact quantum groups, which is just another point of view on the corepresentation theory of the corresponding CQG-algebra. Let us give a definition to make this clear.

\begin{definition}
An \emph{$n$-dimensional representation} of a compact quantum group $\QG$ is an element $v\in M_{n}(\Pol(\QG))$ which is invertible and such that for all $1\leqslant i, j\leqslant n$,
\begin{equation*}
\Delta(v_{ij}) = \sum_{k=1}^{n}v_{ik}\otimes v_{kj}.
\end{equation*}
It is said to be \emph{unitary} if it is unitary as an element of $M_{n}(\Pol(\QG))$.
\end{definition}

Given two representations $v$ and $w$, one can form their direct sum by considering a block diagonal matrix with blocks $v$ and $w$ respectively, and their tensor product by considering the matrix with coefficients
\begin{equation*}
(v\otimes w)_{(i,k),(j,\ell)} = v_{ij}w_{k\ell}.
\end{equation*}
In this setting, an intertwiner between two representations $v$ and $w$ of dimension respectively $n$ and $m$ will be a linear map $T : \C^{n}\to \C^{m}$ such that $Tv = wT$ (we are here identifying $M_{n}(\C)$ with $M_{n}(\C.1_{\Pol(\QG)})\subset M_{n}(\Pol(\QG))$). The set of all intertwiners between $v$ and $w$ will be denoted by $\displaystyle\mathrm{Mor}_{\QG}(v, w)$.

If $T$ is injective, then $v$ is said to be a \emph{subrepresentation} of $w$, and if $w$ admits no non-zero subrepresentation apart from itself, then it is said to be \emph{irreducible}. One of the fundamental results in the representation theory of compact quantum groups is due to S.L. Woronowicz in \cite{woronowicz1995compact} and can be summarized as follows:

\begin{theorem}[Woronowicz]
Any finite-dimensional representation of a compact quantum group splits as a direct sum of irreducible ones, and any irreducible representation is equivalent to a unitary one.
\end{theorem}

\subsection{Definition and basic properties}

Recall that if we are given two compact quantum groups $\QG$ and $\QH$, then we say that $\QH$ is a (closed quantum) subgroup of $\QG$ if there is a surjective Hopf *-algebra morphism
\begin{equation*}
q_{\QH} : \Pol(\QG) \to \Pol(\QH).
\end{equation*}
In that case, $\Pol(\QH)$ is naturally a Hopf-quotient of $\Pol(\QG)$; we will sometimes denote the corresponding Hopf *-ideal by $\I_{\QH}$. As we are dealing with CQG-algebras, each Hopf *-ideal $\I$ of $\Pol(\QG)$ in fact determines a compact quantum group $\QH$ which is a subgroup of $\QG$ such that $\I_{\QH} = \I$. Further note that as $q_{\QH}$ preserves in particular the respective counits, it is easily checked that $q_{\QH}(K_{n}(\Pol(\QG)) = K_{n}(\Pol(\QH))$ for all $n\in \N\cup \{\infty\}$. 

Given a family $(\QH_{i})_{i\in I}$ of quantum subgroups of $\QG$, we define its intersection $\bigwedge_{i\in I} \QH_{i}$ as the quantum subgroup corresponding to the Hopf *-ideal generated by all the $\I_{\QH_{i}}, i \in I$ (which is nothing but their algebraic sum). Conversely the subgroup generated by a given family $(\QH_{i})_{i \in I}$ of quantum subgroups of $\QG$ is defined as the quantum subgroup corresponding to the largest Hopf *-ideal contained in the intersection of the Hopf *-ideals $\I_{\QH_{i}}, i \in I$ (note that an intersection of Hopf ideals is not Hopf in general), and will  be denoted $\bigvee_{i \in I} \QH_i$. All this is discussed in detail for instance in \cite{CHK}.

Let us now move to generating functionals. 

\begin{definition}
Let $\QG$ be a compact quantum group with a quantum subgroup $\QH$. We say that a generating functional $\phi: \Pol(\QG)\to \C$ \emph{factors through} $\QH$ if there exists a functional $\phi_{\QH} : \Pol(\QH) \to \C$ such that $\phi_{\QH} = \phi\circ q_{\QH}$.
\end{definition}

Assume that a generating functional $\phi: \Pol(\QG)\to \C$ factors through $\QH$. Then $\phi_{\QH}$ as above is automatically a generating functional itself (this follows from the facts stated earlier) and moreover the whole Sch\"urmann triple of $\phi_{\QH}$ factors through $\QH$, by which we mean that that if $(\rho_{\QH}, \eta_{\QH}, \phi_{\QH})$ is a (surjective) Sch\"urmann triple for $\phi_{\QH}$, then $(\rho_{\QH}\circ q_{\QH}, \eta_{\QH}\circ q_{\QH}, \phi)$ is a (surjective) Sch\"urmann triple for $\phi$. Further observe that if $\phi$ as above factors through $\QH$ and is Gaussian (respectively, a drift) then $\phi_{\QH}$ is also Gaussian (respectively, a drift).

Note that if $\phi$ as above factors through $\QH$, then it also factors through any quantum subgroup of $\QG$ containing $\QH$. 

\begin{definition} \label{def:Zpart}
Let $\QG$ be a compact quantum group. We define the Gaussian part of $\QG$ to be intersection of all quantum subgroups of $\QG$ through which all Gaussian functionals factor. Note that the intersection is not empty since all functionals factor through $\QG$ itself. The drift part is defined similarly.
\end{definition}

The Gaussian (respectively, drift) part of $\QG$ is determined by the largest Hopf *-ideal of $\Pol(\QG)$ contained in the intersection of all $\textup{Ker}(\phi)$ with $\phi$ Gaussian (respectively, a drift). In other words, it is the smallest quantum subgroup $\QH$ of $\QG$ such that all Gaussian functionals (respectively, drifts) factor through $\QH$. The next result tells us about the behaviour of this construction when we move between a quantum group and its subgroup.

\begin{proposition}\label{prop:ZZ'}
Let $\QG$ be a compact quantum group and let $\QH$ be a quantum subgroup of $\QG$. Then, the Gaussian part of $\QG$ contains the Gaussian part of $\QH$. The same result holds for the drift parts. 
\end{proposition}

\begin{proof}
Let us denote the respective Gaussian parts by $\QK$ and $\QK'$. The latter is viewed as a subgroup of $\QH$ -- hence corresponds to a Hopf *-ideal of $\Pol(\QH)$ that we denote by $\mathcal{I_{\QK'}}$. But we can also view it as a subgroup of $\QG$, and then the corresponding Hopf *-ideal is $q_{\QH}^{-1}(\mathcal{I}_{\QK'})$, as can easily be checked. Thus our claim is equivalent to
the inclusion
\begin{equation*}
\mathcal{I}_{\QK} \subset  q_\QH^{-1}(\mathcal{I}_{\QK'}).
\end{equation*}
But we have already mentioned that 
\begin{equation*}
\bigcap_{\phi \text{ Gaussian}} \textup{Ker} (\phi) \subset q_{\QH}^{-1}\left(\bigcap_{\phi' \text{ Gaussian}} \textup{Ker} (\phi')\right).
\end{equation*}
This ends the proof of the first part, if we note that as $q_{\QH}$ is a surjective Hopf *-morphism, for any set $X\subset \Pol(\QH)$ if $\mathcal{J}$ is the largest Hopf *-ideal contained in $X$, then $q_{\QH}^{-1}(\mathcal{J})$ is the largest Hopf *-ideal contained in $q_{\QH}^{-1}(X)$.

The proof for the drift part is similar.
\end{proof}

\begin{definition}
For a quantum group $\mathbb{G}$, we denote its Gaussian part by $\Gauss(\mathbb{G})$ and for brevity we call $\QG$ Gaussian if $\QG = \Gauss(\mathbb{G})$. 
\end{definition}

\begin{corollary}\label{cor:Gausssubgroups}
If $\QG$ is a compact quantum group with a quantum subgroup $\QH$, then $\Gauss(\mathbb{H})\subset \Gauss(\mathbb{G})$. On the other hand if  $\Gauss(\mathbb{G})$ is contained in $\QH$, then $\Gauss(\mathbb{H})= \Gauss(\mathbb{G})$. Further if $\QG$ is generated by its Gaussian quantum subgroups, it is Gaussian itself.
\end{corollary}

\begin{proof}
Follows immediately from Proposition \ref{prop:ZZ'} and respective definitions.
\end{proof}

\subsection{Kac property of the Gaussian part of a compact quantum group}

A compact quantum group is said to be \emph{of Kac type} if (among many other equivalent characterizations, see for instance \cite[Prop 1.7.9]{nt13}), its antipode is involutive, i.e.\ $S^{2} = \id$. It turns out that this condition is closely linked to gaussianity. To express this, we will use the notion of maximal Kac type quantum subgroup introduced in \cite{soltan2005quantum} but slightly revisited. Recall (see for instance \cite[Sec 1.4]{nt13}) that to any irreducible unitary representation $U=(u_{ij})_{i,j=1}^n$ of $\QG$ one can associate a unique positive invertible operator $Q$ such that $\mathrm{Tr}(Q) = \mathrm{Tr}(Q^{-1})$ and $Q\overline{U}Q^{-1}$ is unitary. It is easy to see that conjugating $U$ by a unitary scalar matrix again yields a unitary representation, hence we may, and will, assume that $Q$ is diagonal. Let $(q_{i})_{1\leqslant i\leqslant n}$ be its diagonal coefficients, i.e., its eigenvalues. Then the formula
\begin{equation*}
\tau_{t}(u_{ij}) = \left(\frac{q_{i}}{q_{j}}\right)^{-it}u_{ij}, \;\;\ t \in \mathbb{R}, i,j=1,\ldots,n,
\end{equation*}
defines consistently a one-parameter group of *-homomorphisms of $\Pol(\QG)$, called the \emph{scaling group} of $\QG$ (see for instance \cite[Prop 1.7.6]{nt13} for a proof).

This can be used to give a new description of the maximal Kac type quantum subgroup of a compact quantum group $\QG$.

\begin{lemma}\label{lem:maximalkac}
Let $\QG$ be a compact quantum group and let $I_{\textup{Kac}}$ be the ideal generated by the range of $S^{2} - \id$. Then, $I_{\text{Kac}}$ is a Hopf *-ideal and the compact quantum group $\QG_{\textup{Kac}}$ defined by $\Pol(\QH) = \Pol(\QG)/I_{\textup{Kac}}$ is the maximal Kac type quantum subgroup of $\QG$.
Moreover we also have that $I_{\text{Kac}}$ is the ideal generated by the  union of the ranges of all $\tau_s - \id$, $s \in \mathbb{R}$.
\end{lemma}

\begin{proof}
Note that the image of $S^{2} - \id$ is a selfadjoint subspace of $\Pol(\QG)$, as follows from the formula $S^{2}\circ\ast = \ast\circ S^{-2}$. Moreover if $x\in \Pol(\QG)$, then using Sweedler's notation we have
\begin{align*}
\Delta(S^{2}(x) - x) & = (S^{2}\otimes S^{2} - \id)\circ\Delta(x) \\
& = \sum S^{2}(x_{(1)})\otimes S^{2}(x_{(2)}) - \sum x_{(1)}\otimes x_{(2)} \\
& = \sum \left(S^{2}(x_{1}) - x_{(1)}\right)\otimes S^{2}(x_{(2)}) + \sum x_{(1)}\otimes \left(S^{2}(x_{(2)}) - x_{(2)}\right) 
\end{align*}
In other words, $\mathrm{Im}(S^{2} - \id)$ is a *-coideal.

Let now $\QK$ be a quantum subgroup of $\QG$ given by $\pi : \Pol(\QG)\to \Pol(\QK)$. Then $\QK$ is of Kac type if and only if $(S^{2} - \id)\circ\pi = 0$. Because, $\pi$ is a Hopf algebra homomorphism, this is equivalent to $\pi\circ(S^{2} - \id) = 0$. In other words, $\QK$ is of Kac type if and only if the range of $S^{2} - \id$ is contained in $\ker(\pi)$. As a consequence, the intersection of all Hopf *-ideals containing the range of $S^{2} - \id$ gives rise to the maximal Kac type quantum subgroup of $\QG$, and that intersection is exactly the ideal $I_{\textup{Kac}}$.

As for the last two statements, we can simply repeat the proof above,  using  the following facts:  for any $t \in \mathbb{R}$ we have that $\tau_t$ is $^*$-preserving and $\Delta \circ \tau_t = (\tau_t \otimes \tau_t)\circ \Delta$, and $\QK$ is of Kac type if and only if $(\tau_s-\id)\circ \pi=0$ for all $s\in \mathbb{R}$.
\end{proof}

We will need hereafter a variant of that characterization, involving an explicit description of the map $S^{2}$. More precisely, the family $(\tau_{t})_{t\in \R}$ has a unique holomorphic extension $(\tau_{z})_{z\in \C}$ which is given exactly by the same formula, and we have the equality $S^{2} = \tau_{-i}$.

\begin{corollary}\label{cor:generatorsKacpart}
Let $\Irr(\QG)$ be a complete set of pairwise inequivalent irreducible unitary represen\-tations of $\QG$ and fix for each $U\in \Irr(\QG)$ an eigenbasis of the corresponding matrix $Q$. Then, with the notations above, the ideal $I_{\textup{Kac}}$ is generated by the set
\begin{align*}
\left\{ u_{ij} \mid U=(u_{ij})_{i,j=1}^n\in \Irr(\QG), i,j=1,\ldots, n, q_{i}\neq q_{j}\right\}.
\end{align*}
\end{corollary}

\begin{proof}
We know by Lemma \ref{lem:maximalkac} that $I_{\textup{Kac}}$ is generated by the range of  $S^{2} - \id$. Because the coefficients of the elements of $\Irr(\QG)$ form a basis of $\Pol(\QG)$, the range of $S^{2} - \id$ is the span of the images of all these coefficients. If $q_{i} = q_{j}$, then $S^{2}(u_{ij}) = u_{ij}$ while if $q_{i}\neq q_{j}$, then
\begin{equation*}
u_{ij} = \left(\frac{q_{i}}{q_{j}} - 1\right)^{-1}(S^{2}(u_{ij}) - u_{ij})\in \mathrm{Im}(S^{2} - \id),
\end{equation*}
hence the result.
\end{proof}

With this we can establish the link between Gaussian processes and the maximal Kac type quantum subgroup.

\begin{theorem}\label{prop-gauss-kac}
Let $\QG$ be a compact matrix quantum group. Any Gaussian process on $\QG$ factors through the maximal Kac type quantum subgroup of $\QG$.
\end{theorem}

\begin{proof}
By Lemma \ref{lem:maximalkac}, we have to prove that $\phi$ vanishes on $I_{\text{Kac}}$. It is in fact enough to prove that it vanishes on $X = \mathrm{Im}(S^{2} - \id)$, i.e. that $\phi$ is $S^{2}$-invariant, thanks to Corollary \ref{cor:wickformula}. 

We first claim that for any Gaussian cocycle $\eta$ we have $\eta\circ(S+\id) = 0$. Indeed, for any $x\in \Pol(\QG)$,
\begin{align*}
0 & = \eta(\varepsilon(x)1) \\
& = \eta(x_{(1)}S(x_{(2)})) \\
& = \eta(x_{(1)})\varepsilon(S(x_{(2)})) + \varepsilon(x_{(1)})\eta(S(x_{(2)})) \\
& = \eta(x_{(1)})\varepsilon(x_{(2)}) + \varepsilon(S(x_{(1)}))\eta(S(x_{(2)})) \\
& = \eta(x) + \eta(S(x))
\end{align*}
where the last step uses the fact that $X$ is a *-coideal. This implies $\eta\circ S^{2}(x) = \eta(x)$ for all $x\in \Pol(\QG)$, hence $\eta$ vanishes on $X$. As a consequence, for any $a\in X$,
\begin{align*}
0 & = \phi(\varepsilon(a)1) = \phi(a_{(1)}S(a_{(2)})) \\
& = \phi(\varepsilon(a_{(1)})S(a_{(2)}))) + \phi(a_{(1)}\varepsilon(S(a_{(2)}))) + \langle\eta(a_{(1)}^{*}), \eta(S(a_{(2)}))\rangle \\
& = \phi(S(a)) + \phi(a) - \langle\eta(a_{(1)}^{*}), \eta(a_{(2)})\rangle \\
& = \phi(S(a)) + \phi(a).
\end{align*}
This implies that for any $a\in X$, $\phi(S^{2}(a)) = a$. Let us now consider an irreducible unitary representation $U=(u_{ij})_{i,j=1}^n$ and consider $i,j=1,\ldots,n$ such that $q_{i}\neq q_{j}$. Then $u_{ij}\in X$ by Corollary \ref{cor:generatorsKacpart}, so that
\begin{equation*}
\phi(u_{ij}) = \phi(S^{2}(u_{ij})) = \frac{q_{i}}{q_{j}}\phi(u_{ij}).
\end{equation*}
This implies $\phi(u_{ij}) = 0$, hence the result.
\end{proof}

\section{Strong connectedness} \label{sec:stronglyconnected}

Before turning to examples, it will be useful to have at hand a condition which is necessary for a compact quantum group to be Gaussian. To introduce it, recall that by definition the intersection of the kernels of all Gaussian functionals contains $K_{3}$. Thus, any Hopf *-ideal contained in $K_{3}$ is contained in $\I_{\Gauss}$. In particular, since $K_{\infty}$ is such an ideal (see below), it must be trivial as soon as $\QG$ is Gaussian. Our purpose in this section is to explore the condition $K_{\infty} = \{0\}$.

\subsection{The definition}

Before going further, let us give a proof of the coideal property for $K_{\infty}$ which was alluded to in the beginning. It does in fact follow from a more general result.

\begin{proposition}\label{prop:hopfideal}
Let $A$ be a Hopf algebra and let $I$ be a Hopf ideal. Then,
\begin{equation*}
I^{\infty} = \bigcap_{k\geqslant 1}I^{k}
\end{equation*}
is a Hopf ideal.
\end{proposition}

\begin{proof}
First observe that $I^{\infty}$ is by construction an ideal which is contained in $I\subset\ker(\varepsilon)$ and invariant under the antipode. Moreover for any $k \in \mathbb{N}$ we have
\begin{equation*}
\Delta(I^{2k})\subset \sum_{n=0}^{2k}I^{n}\otimes I^{2k-n} \subset I^{k}\otimes A + A\otimes I^{k}.
\end{equation*}
As a consequence, we have the inclusion
\begin{equation*}
(I^{k})^\perp \cdot (I^k)^\perp\subset (I^{2k})^\perp
\end{equation*}
of vector subspaces of the dual algebra $A^{*}$ (endowed with the convolution product). This implies that the increasing union
\begin{equation*}
\bigcup_{k\geqslant 1} (I^k)^\perp
\end{equation*}
is a subalgebra of $A^{*}$ and it then follows (from instance from \cite[Thm 2.3.6 (i)]{abe2004hopf}) that
\begin{equation*}
J = \left(\bigcup_{k\geqslant 1} (I^k)^\perp\right)^{\perp} = \left\{x\in A \mid f(x) = 0 \textup{ for all } f\in A^{*} \text{ such that  } \exists_{k\in \mathbb{N}}\, f_{\mid I^{k}} = 0\right\}
\end{equation*}
is a coideal of $A$. We now claim that $I^{\infty} = J$, which is enough to conclude.

Let us first consider $x\in I^{\infty}$. Then, if $f : A\to \C$ is a linear map which vanishes on $I^{k_{0}}$ for some $k_{0}\in \mathbb{N}$, it vanishes on $I^{\infty}$, hence in particular on $x$. In other words, $x\in J$. Conversely, let $x\notin I^{\infty}$. Then, there exists $k_{0}\in \mathbb{N}$ such that $x\notin I^{k_{0}}$. By taking a basis of $I^{k_{0}}$ and completing it into a basis of $A$ containing $x$, we see that there exist a linear map $f : A\to \C$ such that $f(x) = 1$ and $f_{\mid I^{k_{0}}} = 0$. As a consequence, $x$ is not in the kernel of all the elements of $\bigcup_{k\geqslant 1}(I^k)^\perp$, which means that it is not in $J$, concluding the proof.
\end{proof}

Our goal in this section is to investigate compact quantum groups with the property that $K_{\infty} = \{0\}$. To get a better intuition for the meaning of that condition, let us consider the classical case.

\begin{lemma}\label{lem:classicalstronglyconnected}
Let $G$ be a classical compact group. Then $K_{\infty} = \{0\}$ if and only if $G$ is connected.
\end{lemma}

\begin{proof}
It follows from standard Gelfand duality arguments that the closure of $K_{\infty}$ in the C*-algebra $C(G)$ of continuous complex-valued functions on $G$ is the ideal of functions vanishing outside some open subgroup $H\subset G$. Since any open subgroup in a topological group is also closed, it follows that $H$ is a union of connected components of $G$. Let now $Z$ be a connected component of $G$ not containing the neutral element. Then, the indicator function $p = \mathbf{1}_{Z}$ is a continuous function on $G$ and moreover belongs to $K_{1}$. Thus, $p = p^{n}\in K_{n}$ for all $n\in \mathbb{N}$, i.e.\ $p\in K_{\infty}$ so that $Z$ is not contained in $H$. As a conclusion, $H$ is the connected component of the identity and the result follows from the fact that $K_{\infty}=\{0\}$  if and only if its closure equals $\{0\}$.
\end{proof}

In view of this result, we might want to call a compact quantum group $\QG$ connected if $K_{\infty}=\{0\}$. However, there is already a notion of connectedness in the literature, introduced by Wang in \cite{wang2009simple} and studied in detail in \cite{connected14}: a compact quantum group is said to be \emph{connected} if $\Pol(\QG)$ does not contain any finite-dimensional Hopf *-subalgebra. It turns out that vanishing of $K_{\infty}$ is stronger than this (see Proposition \ref{prop:stronglyconnected} below), hence we choose the following terminology.

\begin{definition}
A compact quantum group $\QG$ is said to be \emph{strongly connected} if $K_{\infty} = \{0\}$.
\end{definition}

To see the link with the aforementioned definition of connectedness, we need to generalize the idea concerning projections used in the proof of Lemma \ref{lem:classicalstronglyconnected}.

\begin{lemma}\label{lem:projectionsnotconnected}
Let $\QG$ be a compact quantum group and let $p\in \Pol(\QG)$ be a non-trivial projection. Then, $p-\varepsilon(p)1\in K_{\infty}(\Pol(\QG))\neq\{0\}$.
\end{lemma}

\begin{proof}
Because the counit $\varepsilon$ is a *-homomorphism, $\varepsilon(p)$ is an idempotent in $\C$, hence equals $0$ or $1$. In both cases, $q = p - \varepsilon(p)\in K_{1}$ and satisfies $q^{2}\in \{q, -q\}$. Thus, $q = \pm q^{n} \in K_{n}$ for all $n$, hence $q \in K_{\infty}$.
\end{proof}

We are now ready to state and prove several fundamental properties of strong connectedness. Note that for duals of discrete groups, this has already been studied as the residual nilpotency of the augmentation ideal of the group ring, and the monograph \cite{passi1979group} gives a comprehensive survey of the known results to which we will refer.

\begin{proposition}\label{prop:stronglyconnected}
Let $\QG$ be a compact quantum group. The following hold:
\begin{enumerate}
\item If $\QG$ is the dual of a discrete group $\Gamma$, then it is strongly connected if and only if $\Gamma$ is residually `torsion-free nilpotent' (recall that it means that for any non-identity element $\gamma \in \Gamma$, there is a normal subgroup $N$ of $\Gamma$ such that $\gamma\notin N$ and $\Gamma/N$ is torsion-free nilpotent);
\item If $\QG$ is strongly connected, then it is connected but the converse does not hold in general;
\item If $\QG$ is topologically generated (in the sense of \cite{chirvasitu20}) by strongly connected quantum subgroups, then it is strongly connected.
\end{enumerate}
\end{proposition}

\begin{proof}
\begin{enumerate}
\item This is the contents of \cite[Thm VI.2.26]{passi1979group}.
\item Assume that  $\QG$ is not connected and $\Pol(\QH)\subset \Pol(\QG)$ is a finite-dimensional Hopf *-subalgebra. As $\Pol(\QH)$  is a finite-dimensional $C^*$-algebra, it contains non-zero projections, hence the result follows from Lemma \ref{lem:projectionsnotconnected}.

For the dual of a discrete group, connectedness is equivalent to torsion-freeness. Hence, the dual of any torsion-free group which is not residually nilpotent is connected but not strongly connected.
\item Assume that $\QG$ is topologically generated by quantum subgroups $(\QH_{i})_{i\in I}$ and let us consider the corresponding surjective $*$-homomorphisms $\pi_{i} : \Pol(\QG)\to\Pol(\QH_{i})$. Obviously, $\pi_{i}(K_{\infty}(\Pol(\QG)))\subset K_{\infty}(\Pol(\QH_{i}))$ so that under the assumption of the statement,
\begin{equation*}
K_{\infty}(\Pol(\QG))\subset \bigcap_{i\in I}\ker(\pi_{i}).
\end{equation*}
By definition of topological generation, the right-hand side does not contain any non-zero Hopf *-ideal, hence the result.
\end{enumerate}
\end{proof}

\begin{remark}
The first point in the previous proposition suggests that there may be a connection between strong connectedness of a compact quantum group and torsion-freeness of its dual discrete quantum groups. There is however no relationship with torsion-freeness in the sense of Meyer \cite{meyer2008homological}. Indeed,
\begin{enumerate}
\item $SO(N)$ is connected but not simply connected, hence it is strongly connected but its dual is not torsion free (it has a projective representation comming from its universal covering which yields an ergodic finite-dimensional action not equivariantly Morita equivalent to the trivial one);
\item If $\Gamma$ is a torsion-free group which is not residually nilpotent, then it is torsion-free while its dual is not strongly connected.
\end{enumerate}
\end{remark}

The topological generation criterion can prove useful to provide examples which are neither commutative nor cocommutative.

\begin{proposition}
The free unitary quantum group $U_{N}^{+}$ is strongly connected for all $N\in \mathbb{N}$.
\end{proposition}

\begin{proof}
For $N = 1$, $U_{N}^{+}$ is just the circle, which is strongly connected because it is connected. For $N\geqslant 2$, the result immediately follows from the following two facts:
\begin{itemize}
\item $U_{N}^{+}$ is topologically generated by $U(N)$ and $\mathbb{F}_{N}$ by \cite{chirvasitu20};
\item $K_{\infty}(\Pol(U(N))) = \{0\}$ by connectedness while $K_{\infty}(\C[\mathbb{F}_{N}]) = \{0\}$ as $\mathbb{F}_{N}$ is residually `torsion-free nilpotent' by a result of Magnus \cite{magnus1935beziehungen}.
\end{itemize}
\end{proof}

For $O_{N}^{+}$ one cannot apply directly the same strategy, because the diagonal quotient is $\Z_{2}^{\ast N}$, which is not torsion free. To get a better insight, let us deal with the very special case $N=2$.

\begin{lemma}
The strongly connected component of the identity of $O_{2}^{+}$ is the circle $\T$.
\end{lemma}

\begin{proof}
Recall that $O_{2}^{+}$ is isomorphic to $SU_{-1}(2)$, so that $\Pol(O_{2}^{+})$ is generated by two elements $\alpha$ and $\gamma$ such that $\gamma$ is normal, anti-commutes with $\alpha$ and 
\begin{equation*}
\alpha\alpha^{*} + \gamma\gamma^{*} = 1.
\end{equation*}
Moreover, $\gamma\in K_{1}(\Pol(O_{2}^{+}))$ and $\beta = \alpha - 1\in K_{1}(\Pol(O_{2}^{+}))$.
Now, the equation above yields
\begin{equation*}
\beta + \beta^{*} = \beta^{*}\beta + \gamma^{*}\gamma\in K_{2}
\end{equation*}
while the anti-commutation relations translate into
\begin{equation*}
\gamma(\beta + \beta^{*}) = -(\beta + \beta^{*})\gamma - 4\gamma
\end{equation*}
so that
\begin{equation*}
\gamma = \frac{-1}{4}\left(\gamma(\beta + \beta^{*}) + (\beta + \beta^{*})\gamma\right).
\end{equation*}
As a consequence, if $\gamma\in K_{n}$ then $\gamma\in K_{n+1}$. Since $\gamma\in K_{1}$, it follows by a straightforward induction that $\gamma\in K_{\infty}$. Thus, $\langle \gamma\rangle\subset K_{\infty}$. On the other hand, if $\pi$ denotes the surjection onto $\Pol(\T) = \Pol(O_{2}^{+})/\langle\gamma\rangle$, we have
\begin{equation*}
\pi(K_{\infty}(\Pol(O_{2}^{+}))\subset K_{\infty}(\Pol(\T)) = \{0\},
\end{equation*}
so that $K_{\infty}\subset \langle\gamma\rangle$, concluding the proof.
\end{proof}

\subsection{Totally strongly disconnected quantum groups}

It is quite natural to consider the Hopf *-algebra $\Pol(\QG)/K_{\infty}$ to be the function algebra on the strongly connected component of the identity. We will therefore denote it by $\Pol(\QG^{00})$, the double $0$ exponent being meant to distinguish it from the connected component of the identity $\QG^{0}$ in the sense of \cite{connected14}. Of course, if $G$ is classical then this coincides with the usual connected component of the identity. Another extreme case, in contrast to being strongly connected, is when $\QG^{00}$ reduces to the trivial group.

\begin{definition}
A compact quantum group is said to be \emph{totally strongly disconnected} if $\QG^{00}$ is trivial, i.e.\ if $K_{\infty} = K_{1}$.
\end{definition}

\begin{remark}
One may wonder whether the property of being totally strongly disconnected can be expressed in terms of representation theory. We do not know, but we can at least mention that it is not a property of the corresponding C*-tensor category since it is not preserved under monoidal equivalence. Indeed, we will see below that $S_{N}^{+}$ is totally strongly disonnected, while the quantum automorphism group of $(M_{N}(\C), \mathrm{tr})$, which is monoidally equivalent to $S_{N^{2}}^{+}$, is strongly connected (because $U_{N}^{+}$ is).
\end{remark}

By definition, $\QG^{00}$ is strongly connected, hence connected. As a consequence, it is a quantum subgroup of the connected component of the identity $\QG^{0}$ in the sense of \cite[Def 4.11]{connected14}. This implies that if $\QG$ is totally disconnected in the sense of \cite{connected14}, then it is also totally strongly disconnected.

Before going further, let us describe that property for duals of discrete groups. Interestingly, there turns out to be a simple characterization. We will denote by $\gamma_{2}(\Gamma)$ the subgroup of $\Gamma$ generated by commutators and by $\sqrt{\gamma_{2}(\Gamma)}$ the group of all elements of $\Gamma$ of which a finite power lies in $\gamma_{2}(\Gamma)$.

\begin{proposition}
Let $\Gamma$ be a discrete group. Then $\widehat{\Gamma}$ is totally strongly disconnected if and only if its abelianization is torsion.
\end{proposition}

\begin{proof}
Obviously, $K_{\infty} = K_{1}$ if and only if $K_{2} = K_{1}$ if and only if $K_{2} + 1 = \C[\Gamma]$. Taking intersections with $\Gamma$ in the last equality then yields by \cite[Thm IV.1.5]{passi1979group} the equivalent condition $\sqrt{\gamma_{2}(\Gamma)} = \Gamma$, which in turn means that the abelianization of $\Gamma$ is torsion.
\end{proof}

\begin{remark}
A group has torsion abelianization if and only if it has no non-zero homomorphism to $\mathbb{Q}$. This can in turn be restated in a homological way by saying that the first Betti number of $\Gamma$ vanishes, or that $H_{1}(\Gamma, \mathbb{Q}) = \{0\}$.
\end{remark}

Totally strongly disconnected compact quantum groups are interesting to us because they have, by definition, trivial Gaussian part. Here is a sufficient criterion for total strong disconnectedness, which will yield our first examples of computation of Gaussian parts.

\begin{proposition}\label{prop:projections}
If $\QG$ is a compact quantum group such that $\Pol(\QG)$ is generated by projections, then it is strongly totally disconnected.
\end{proposition}

\begin{proof}
Assume that $\Pol(\QG)$ is generated by projections $(p_{i})_{i\in I}$. By Lemma \ref{lem:projectionsnotconnected}, $p_{i} - \varepsilon(p_{i})\in K_{\infty}$ for all $i\in I$ and since these elements generate $K_{1}$, the proof is complete.
\end{proof}

\begin{remark}
The quantum permutation group $S_{N}^{+}$ is is totally strongly disconnected in our sense while it is connected in the sense of \cite{connected14}. This shows that strong connectedness is strictly stronger than connectedness.
\end{remark}

Some straightforward examples to which the previous statement applies are the following:
\begin{itemize}
\item Finite quantum groups, i.e.\ those for which the corresponding Hopf algebra is finite-dimensional;
\item Quantum permutation groups, i.e.\ quantum subgroups of $S_{N}^{+}$ for $N \in \mathbb{N}$ (this includes for instance the quantum reflection groups $H_{N}^{s+}$ for all $1\leqslant s < +\infty$);
\item Profinite compact quantum groups in the sense of \cite{connected14}, or in the terms of \cite{CaspersSkalski}, duals of locally finite discrete quantum groups. Note also that \cite{connected14} gives examples of duals of discrete groups which are totally disconnected as compact quantum groups, hence totally strongly disconnected, but not profinite.
\end{itemize}

\begin{remark}
It is natural to wonder whether when $\QG$ is totally strongly disconnected, then $\Pol(\QG)$ is generated by projections. A negative answer can be provided by an example of a discrete group $\Gamma$ with torsion abelianization such that $\C[\Gamma]$ does not contain any non-trivial projection. For instance, the group whose existence in given in \cite[Cor 3.2]{grigorchuk2016constructions} is torsion-free (because it is orderable) and amenable (because it is locally solvable) hence satisfies the Kadison-Kaplansky conjecture (there is no non-trivial idempotent in $\C[\Gamma]$) but has trivial abelianization since it is perfect.
\end{remark}

\subsection{Link to Kac type quantum groups}

Even though the definition of strong connectedness is very general, it turns out that it entails a strong restriction on compact quantum groups. More precisely, a strongly connected compact quantum group must be of Kac type.

\begin{proposition}
A strongly connected compact quantum group is of Kac type.
\end{proposition}

\begin{proof} 
Let $\QG$ be a compact quantum group. Following the notations of Corollary \ref{cor:generatorsKacpart}, let $U \in \Irr(\QG)$ be an $n$-dimensional unitary representation of $\QG$.

For $j, k = 1,\cdots, n$, set $\hat{u}_{jk} = u_{jk}-\delta_{jk}1 \in K_1$. 
Then if $j \neq k$ we have
\begin{equation*}
0 = \sum_{\ell=1}^{n} u_{j\ell} u^{*}_{k\ell} = \sum_{\ell=1}^{n} \hat{u}_{j\ell} \hat{u}^{*}_{k\ell} + u_{jk} + u^{*}_{kj}
\end{equation*}
and
\begin{equation*}
0 = \sum_{\ell=1}^{n} \frac{q_{k}}{q_{\ell}} u_{\ell k} u^{*}_{j\ell} = \sum_{\ell=1}^{n} \frac{q_k}{q_\ell} \hat{u}_{\ell k} \hat{u}^{*}_{\ell j} + \frac{q_{k}}{q_{j}} u_{jk} + u^{*}_{kj}.
\end{equation*}
Taking the difference, we get
\begin{equation*}
\left(\frac{q_{k}}{q_{j}} - 1\right)u_{jk} = \sum_{\ell=1}^{n} \hat{u}_{j\ell} \hat{u}^{*}_{k\ell} - \sum_{\ell=1}^{n} \frac{q_{k}}{q_{\ell}} \hat{u}_{\ell k} \hat{u}^{*}_{\ell j}.
\end{equation*}
This shows that if $q_{j}\neq q_{k}$ then $u_{jk}=\hat{u}_{jk}\in K_{2}$, and by taking adjoints, $u^{*}_{jk}= \hat{u}^*_{jk}\in K_{2}$.
If we now look again at the expression above, still assuming $q_{j}\neq q_{k}$, we see that each element of the sum on the right hand side is a product of elements in $K_1$ and $K_2$ (this follows, as whenever $q_{j}\neq q_{k}$ we necessarily have for any $\ell=1,\ldots,n$ that either $q_{\ell}\neq q_{j}$ or $q_{\ell}\neq q_{k}$). Thus in fact $u_{jk}=\hat{u}_{jk}\in K_{3}$, and similarly $u_{jk}^*=\hat{u}_{jk}^*\in K_{3}$. By a straightforward induction, we conclude that $u_{jk}, u^{*}_{jk}\in K_{\infty}$ whenever $q_{j}\neq q_{k}$.

By Corollary \ref{cor:generatorsKacpart}, the elements above generate $I_{\text{Kac}}$ as an ideal, hence $I_{\text{Kac}}\subset K_{\infty}$.
\end{proof}

As noted in the beginning of the section, because $K_{\infty}\subset K_{3}$, any Gaussian functional factors through the strongly connected part. Combining that observation with the previous statement yields Theorem A of the introduction.

\section{Examples I}  \label{sec:ExamplesI}
In this section we fully characterise the Gaussian parts of classical compact groups and duals of classical discrete groups.

\subsection{Classical compact groups}

We will now determine the Gaussian part of any classical compact group. We will in fact show a stronger result concerning the drift part of any compact quantum group. This will be done in two steps, the first being a reduction to the classical case. To do so, let us say that for a compact quantum group $\QG$, its \emph{classical version} is the quantum subgroup corresponding to the Hopf *-ideal of commutators (see for instance \cite{Daws}).

\begin{proposition}\label{prop:drifts}
Let $\QG$ be a compact quantum group. Then, its drift part is the closed subgroup of its classical version generated by all continuous one-parameter subgroups.
\end{proposition}

\begin{proof}
We start by observing that since $[a, b] = [a-\varepsilon(a), b - \varepsilon(b)]$, the Hopf *-ideal of commutators of $\Pol(\QG)$ is contained in $K_{2}$. Hence, any drift vanishes on it so that if $\QH$ denotes the drift part of $\QG$, then $\QH$ is contained in the quantum group obtained by quotienting $\Pol(\QG)$ with the commutators, i.e.\ in the classical version of $\QG$.

Because characters on $\Pol(G)$ correspond to elements of $G$, Remark \ref{rem:drift} implies that drifts are in one-to-one correspondence with continuous one-parameter subgroups. Now if $\phi$ is a drift and $H$ is a closed subgroup of $G$, then as soon as $\phi$ factors through the restriction map $\Pol(G)\to \Pol(H)$, the corresponding one-parameter subgroup lives in $H$. Thus, any closed subgroup through which all drifts factor must contain the closed subgroup $K$ generated by all continuous one-parameter subgroups. Since $K$ is obviously its own drift part, the proof is complete.
\end{proof}

\begin{remark}
There is one subtlety in the statement, which is that the drift part is by definition a closed subgroup. Hence, the result states that the subgroup generated by one-parameter subgroups in a connected compact group is dense, but not that it is itself equal to the whole group. Indeed, solenoids such as the Pontryagin dual of the additive group of rationals have a unique one-parameter subgroup which is not closed but which is dense.
\end{remark}

We are now left with characterizing the drift part of a classical compact group, and this is easily done using the notion of generalized Lie algebra.

\begin{proposition}\label{prop:driftclassical}
The drift part of a compact group $G$ is the connected component of the identity.
\end{proposition}

\begin{proof}
We will prove the equivalent statement that $G$ equals its drift part if and only if it is connected. One way is clear: the subgroup generated by one-parameter subgroups is path-connected, hence also connected, so that its closure is connected. As for the other direction, recall that any connected compact group is a projective limit of connected compact Lie groups. As a consequence, these are LP-groups in the sense of \cite[Def 3.1]{lashof1957lie}. It then follows from \cite[Thm 3.5]{lashof1957lie} that the connected component of the identity is the closure of the range of the exponential map on the generalized Lie algebra and is therefore contained in the drift part.
\end{proof}

\begin{corollary}\label{cor:GaussClassical}
The Gaussian part of a classical compact group is the connected component of the identity.
\end{corollary}

\begin{proof}
It follows from Proposition \ref{prop:driftclassical} that if $G$ is connected, then it equals its drift part, which is itself contained in the Gaussian part, so that any connected compact group is Gaussian. Conversely, if $G$ is Gaussian then it is strongly connected, hence connected by Lemma \ref{lem:classicalstronglyconnected}.
\end{proof}

\subsection{Duals of discrete groups}

The next case to consider is that of duals of discrete groups. Let us recall that for a group $\Gamma$, one defines its canonical lower central series by setting $\gamma_{1}(\Gamma) = \Gamma$ and for any $n \in \mathbb{N}$ putting $\gamma_{n+1}(\Gamma)$ to be the subgroup generated by $[\gamma_{n}(\Gamma), \Gamma]$. Moreover, for a subgroup $\Lambda\subset \Gamma$, we write $\sqrt{\Lambda}$ for the subgroup of all elements $\gamma\in \Gamma$ such that there exists  $n\in \mathbb{N}$ satisfying $\gamma^{n}\in \Lambda$.

Before proving the main result of this section, we need to clarify the connection between group commutators and the structure of the group algebra.

\begin{lemma}\label{lem:groupalgebracommutators}
Let $\Gamma$ be a discrete group and let $g, h, k\in \Gamma$. Then,
\begin{equation*}
[[g, h], k] - 1\in K_{3}(\C[\Gamma]).
\end{equation*}
\end{lemma}

\begin{proof}
For clarity, we will denote by $[\cdot, \cdot]_{\Gamma}$ the group commutator and by $[\cdot, \cdot]_{\C[\Gamma]}$ the algebra commutator. Start by noticing that for $g, h\in \Gamma$,
\begin{equation*}
[g, h]_{\Gamma} - 1 = [g, h]_{\C[\Gamma]}g^{-1}h^{-1}.
\end{equation*}
It then follows that
\begin{align*}
[[g, h], k]_{\Gamma} - 1 & = [[g, h]_{\Gamma}, k]_{\C[\Gamma]}[g, h]_{\Gamma}^{-1}k^{-1} \\
& = [[g, h]_{\C[\Gamma]}g^{-1}h^{-1}, k]_{\C[\Gamma]}[g, h]_{\Gamma}^{-1}k^{-1}.
\end{align*}
Now observe that for any $a, b\in \C[\Gamma]$,
\begin{equation*}
[a, b]_{\C[\Gamma]} = [a-\varepsilon(a), b-\varepsilon(b)]_{\C[\Gamma]}
\end{equation*}
so that in particular all algebra commutators are in $K_{2}$ and if $a\in K_{2}$, then the commutator is in $K_{3}$. The result then follows.
\end{proof}

We are now ready for the characterization of Gaussianity for duals of discrete groups. Recall that a group is said to be \emph{nilpotent of class $2$} if all its commutators are central.

\begin{theorem}\label{dual:Gaussian}
Let $\Gamma$ be a finitely generated discrete group. Then,
\begin{equation*}
\Gauss(\widehat{\Gamma}) = \widehat{\Gamma/\sqrt{\gamma_{3}(\Gamma)}}.
\end{equation*}
In particular, $\widehat{\Gamma}$ is Gaussian if and only if $\Gamma$ is torsion-free nilpotent of class $2$.
\end{theorem}

\begin{proof}
Assume first that $\widehat{\Gamma}$ is Gaussian. The subgroup $\gamma_{3}(\Gamma)$ being normal, the set $\C[\gamma_{3}(\Gamma)] - 1$ is a Hopf *-ideal in $\C[\Gamma]$. Because it is contained in $K_{3}$ by Lemma \ref{lem:groupalgebracommutators}, all Gaussian functionals vanish on it by definition so that they factor through $\widehat{\Gamma/\gamma_{3}(\Gamma)}$. Thus, $\gamma_{3}(\Gamma)$ must be trivial so that $\Gamma$ is nilpotent of class $2$. Let furthermore $g\in \Gamma\backslash\{e\}$ and $n \in \mathbb{N}$ be such that $g^n=e$ Then,
\begin{equation*}
p_{g} = \frac{1}{n}\sum_{k=0}^{n-1}g^{k}
\end{equation*}
is a non-trivial projection in $\C[\Gamma]$, hence $\widehat{\Gamma}$ is not strongly connected by Lemma \ref{lem:projectionsnotconnected}, contradicting Gaussianity. Therefore, we have proven that $\Gamma$ is torsion-free nilpotent of class $2$.

This means that for an arbitrary discrete group $\Gamma$, the Gaussian part of $\widehat{\Gamma}$ is of the form $\widehat{\Lambda}$ with $\Lambda$ a torsion-free nilpotent of class $2$ quotient of $\Gamma$. There is one such quotient which is maximal in the sense that all the other ones factor through it, namely $\Gamma/\sqrt{\gamma_{3}(\Gamma)}$. To conclude it is enough to prove that the duals of all such groups are Gaussian.
To do this, note first that $\Lambda/Z(\Lambda)$ is an abelian finitely generated  group. By \cite{Malcev}, $\Lambda/Z(\Lambda)$ is also torsion free, thus 
\begin{equation*}
\widehat{\Lambda/Z(\Lambda)} \cong \T^{n}
\end{equation*}
for some $n\in \mathbb{N}_{0}$ and $\widehat{\Lambda/Z(\Lambda)}$ is Gaussian because it is a connected compact group. By Corollary \ref{cor:Gausssubgroups}, this means that $\Gauss(\widehat{\Lambda})$ contains $\widehat{\Lambda/Z(\Lambda)}$, or in other words that $\Gauss(\widehat{\Lambda}) = \widehat{\Lambda/\Theta}$ where $\Theta < Z(\Lambda)$. 

It remains then to show that for every $\gamma_{0}\in Z(\Lambda)\setminus \{e\}$ there is a Gaussian generating functional on $\mathbb{C}[\Lambda]$ -- i.e.\ conditionally positive-definite function $\phi$ on $\Lambda$ of Gaussian type  -- such that $\phi(\gamma_{0})\neq 0$. But once again we have $Z(\Lambda)\cong \mathbb{Z}^{m}$ for some $m\in \N_{0}$ and using the usual Laplacian on $\T^{m}$ we first get a function $\phi_{0} : Z(\Lambda)\to \C$ as above by setting
\begin{equation*}
\phi_{0}(k_{1}, \cdots, k_{m}) = \sum_{i=1}^m k_{i}^{2},
\end{equation*}
and then extend it to $\Lambda$ as follows:
\begin{equation*}
\phi(\gamma) = \begin{cases}\phi_{0}(\gamma) & \gamma\in Z(\Lambda) \\ 
0 & \gamma \notin Z(\Lambda)\end{cases}.
\end{equation*}
An elementary check shows that $\phi$ is a conditionally positive-definite function of Gaussian type.
\end{proof}

\begin{remark}
In the case of group algebras, it turns out that $K_{3}$ is a Hopf ideal, and it follows from \cite[Thm IV.1.5]{passi1979group} that it equals $\C[\sqrt{\gamma_{3}(\Gamma)}] - 1$. We could have used this (involved) result to prove directly that a discrete group dual is Gaussian only if it is torsion-free nilpotent of class 2. We have chosen however to give a direct and self-contained proof which is furthermore completely elementary.
\end{remark}

Let us conclude this section with a word on free wreath products of a discrete group by the quantum permutation group $S_{N}^{+}$. We refer the reader to \cite{bichon2004free} for the definition of these objects, whose Gaussian part can be easily expressed in terms of the building discrete group.

\begin{proposition}
Let $N\in \mathbb{N}$ and let $\Gamma$ be a discrete group. Then
\begin{equation*}
\Gauss(\widehat{\Gamma}\wr_{\ast} S_{N}^{+}) = \Gauss(\widehat{\Gamma}^{\ast N}).
\end{equation*}
\end{proposition}

\begin{proof}
Recall that we have a generating family of $N$-dimensional representations $\{u(g):g \in \Gamma\}$ whose coefficients satisfy in particular the relations
\begin{equation*}
u_{ij}(g)u_{ik}(h) = \delta_{jk}u_{ij}(gh), 
\end{equation*}
and
\begin{equation*}
\varepsilon(u(g)_{ij}) = \delta_{ij}
\end{equation*}
for all $g, h\in \Gamma$ and all $1\leqslant i,j\leqslant N$. It follows that for any $g \in \Gamma$ and $i,j =1,\ldots, n$, $i\neq j$,
\begin{equation*}
u_{ij}(g) = u_{ij}(g)(u_{ii}(g) - 1)\in K_{2}
\end{equation*}
and by induction, all the off-diagonal coefficients are in $K_{\infty}$. As a consequence, the Gaussian part factors through the diagonal quantum subgroup, which is the dual of $\Gamma^{\ast N}$.
\end{proof}

In particular, the Gaussian part of $H_{N}^{\infty +}$ is the free residually ``torsion-free nilpotent'' group of rank $N$.

\section{Examples II} \label{sec:ExamplesII}

In this, last section we describe Gaussian parts of $q$-deformations, and discuss the case of free quantum groups.

\subsection{$q$-deformations}

Consider a simply connected semisimple compact Lie group $G$ equipped with a standard Poisson structure and the deformation parameter $q\in (0,1)$. The quantisation procedure due to Korogodski and Soibelman \cite{ks98}, see also \cite{nt12}, leads to a compact quantum group $\QG_q$. The procedure is compatible with the deformation of Poisson subgroups of $G$; in particular the maximal torus, i.e.\ a maximal abelian connected subgroup of $G$, which is unique up to a  conjugation, remains (a classical) subgroup of $\QG_q$.

\begin{proposition}\label{prop:q-deform}
For a simply connected semisimple compact Lie group $G$ and $q\in (0,1)$ the Gaussian part of the quantum group $\QG_q$ is the maximal torus $\mathbf{T}\subset  {\QG_q}$.
\end{proposition}

\begin{proof}
Lemma 4.10 in \cite{tom07}, based on the knowledge of the representation theory of $\Pol(\QG_q)$, shows that the maximal torus $\mathbf{T} $ coincides with the Kac part of $\QG_q$. As $\mathbf{T}$ is by definition connected, the proof is complete.
\end{proof}

\subsection{Deformations of the orthogonal group}

We will now consider the \emph{half-liberated quantum orthogonal group} $O_{N}^{*}$ introduced in \cite{banica2009liberation}. Recall that $\Pol(O_{N}^{*})$ is defined to be the quotient of $\Pol(O_{N}^{+})$ by the relations $abc = cba$ for all $a, b, c\in\{u_{jk} \mid 1\leqslant j, k\leqslant N\}$. We will prove that its Gaussian part is that of the classical group $O_{N}$. This requires first characterizing Gaussian functionals on $O_{N}^{+}$ lying in its classical part. More precisely, any Gaussian functional on $O_{N}^{*}$ yields a Gaussian functional on $O_{N}^{+}$ by composition with the quotient map, and it is easy to determine when such a Gaussian functional in fact factors through the abelianization.

\begin{lemma}\label{lem:gaussianOn}
Let $N \in \mathbb{N}$.
A Gaussian triple on $O_{N}^{+}$ factors through $O_{N}$ if and only if the corresponding cocycle $\eta$ satisfies
\begin{equation*}
\langle\eta(a),\eta(b)\rangle\in\R \qquad \forall a, b\in\{u_{jk} \mid 1\leqslant j, k\leqslant N\}.
\end{equation*}
\end{lemma}

\begin{proof}
This follows immediately from the equivalence of (ii) and (iv) in Remark \ref{remclass}.
\end{proof}

We can now elucidate the Gaussian part of $O_{N}^{*}$.

\begin{proposition} \label{prop:halfliberated}
	Let $N \in \mathbb{N}$.
The Gaussian part of $O_{N}^{*}$ is $SO_{N}$.
\end{proposition}

\begin{proof}
First note that for any Gaussian generating functional we have
\begin{align*}
\phi(abc) & = \varepsilon(ab)\phi(c) + \varepsilon(ac)\phi(b) + \varepsilon(bc) \phi(a)
\\
& + \langle(\eta(a^{*}), \eta(b)\rangle \varepsilon(c) + \langle\eta(a^{*}), \eta(c)\rangle \varepsilon(b) + \langle\eta(b^{*}), \eta(c)\rangle\varepsilon(a).
\end{align*}
Thus, if $j,k=1,\ldots, N$ and  we have $\eta(u_{jk}) = v_{jk}\in\mathbb{C}$, then the condition $\phi(abc) = \phi(cba)$ for $a, b, c\in\{u_{jk} \mid 1\leqslant j, k\leqslant N\}$ becomes
\begin{equation*}
\delta_{jk} \overline{v_{\ell m}} v_{np} + 
\delta_{\ell m} \overline{v_{jk}} v_{np} + 
\delta_{np} \overline{v_{jk}} v_{\ell m} 
=
\delta_{np} \overline{v_{\ell m}} v_{jk} + 
\delta_{\ell m} \overline{v_{np}} v_{jk} + 
\delta_{jk} \overline{v_{np}} v_{\ell m} 
\end{equation*}
for all $j, k, \ell, m, n, p\in\{1\, \ldots, N\}$.
Now, by \cite[Proposition 3.7]{das+al18}, $v = (v_{jk})_{j,k=1}^N$ has to be anti-symmetric, therefore $v_{jj} = 0$ for $j = 1, \ldots, N$.
We have several possibilities:
\begin{itemize}
\item If two or three of the index pairs contain twice the same index, e.g.\ $j=k$ and $\ell=m$, then the all terms vanish;
\item If none of the index pairs contains twice the same index, i.e.\ if $j\neq k$, $\ell\neq m$, and $n\neq p$, then we also find $0 = 0$, thanks to the Kronecker symbols;
\item The only non-trivial case occurs if only one index pair contains twice the same index. Without loss of generality we can assume that this is the first pair, i.e.\ $j=k$. In that case we find the condition
\begin{equation*}
\overline{v_{\ell m}} v_{np} = \overline{v_{np}} v_{\ell m}
\end{equation*}
for all $\ell, m, n, p\in\{1, \ldots, N\}$.
\end{itemize}
The last condition exactly means that $\langle \eta(u_{\ell m}), \eta(u_{np})\rangle \in \R$, so that we can conclude by Lemma \ref{lem:gaussianOn} that all Gaussian triples factor through $O_{N}$. As a consequence,
\begin{equation*}
\Gauss(O_{N}^{*}) = \Gauss(O_{N}) = SO_{N}.
\end{equation*}
\end{proof}

We can deal in a similar way with another variant of $O_{N}^{+}$ called the \emph{twisted orthogonal quantum group} $\overline{O}_{N}$ and introduced in \cite{bb07} as a type of $q$-deformation of orthogonal groups at $q = -1$.
More precisely, we let $\Pol(\overline{O}_{N})$ be the quotient of $\Pol(O_{N}^{+})$ by the relations :
\begin{equation*}
u_{ij}u_{k\ell} = \left\{\begin{array}{ccc}
-u_{k\ell}u_{ij} & \text{if} & i = k \:\&\: j\neq \ell \text{ or } i \neq k \:\&\: j = \ell \\
u_{k\ell}u_{ij} & \text{otherwise}
\end{array}\right.
\end{equation*}
In other words, the generators anti-commute if they are on the same row or column and commute otherwise.

\begin{proposition}
The Gaussian part of $\overline{O}_{N}$ is trivial for any $N \in \mathbb{N}$.
\end{proposition}

\begin{proof}
The computations of Lemma \ref{lem:gaussianOn} show that the condition for the second case is equivalent to $\overline{v_{ij}}v_{k\ell}\in \R$ for any $i,j,k,l=1,\ldots, N$ as soon as $i\neq k$ and $j\neq \ell$. As for the first case, observe that
\begin{equation*}
\phi(u_{ij}u_{i\ell} + u_{i\ell}u_{ij}) = 2(\delta_{ij}\phi(u_{i\ell}) + \delta_{i\ell}\phi(u_{ij})) + 2\mathrm{Re}\left(\overline{v_{ij}}v_{i\ell}\right).
\end{equation*}
Recalling that for any Gaussian functional on $O_{N}^{+}$, $v$ is anti-symmetric (see \cite[Section 3]{das+al18}), we get for $\ell\neq j = i$
\begin{equation*}
\phi(u_{i\ell}) = \mathrm{Re}(\overline{v_{ii}}v_{i\ell}) = 0
\end{equation*}
so that $\phi$ factors through the quotient by the ideal generated by all the off-diagonal coefficients. The quotient is easily seen to be the dual of $\Z_{2}^{\times N}$, which has trivial Gaussian part because it is totally strongly disconnected, hence the result.
\end{proof}

\subsection{Free quantum groups}\label{subsec:free}

\begin{proposition}\label{prop:freeqg}
For $N\geqslant 4$, the Gaussian part of $O_{N}^{+}$ is neither classical nor dual to a discrete group. The same holds for $U_{N}^{+}$ with $N\geqslant 2$.
\end{proposition}

\begin{proof}
Let us first note the Gaussian part of the dual of the free group $\mathbb{F}_{2}$ on two generators is not classical. Indeed, it arises from the quotient of $\mathbb{F}_{2}$ obtained by the relations making the commutator of the two generators central, which is easily seen to be isomorphic to the discrete Heisenberg group. As a consequence, $O_{4}^{+}$ contains, through the quotient $\Pol(O_{4}^{+})\to \Pol(O_{2}^{+}\ast O_{2}^{+})$, the dual of the (non-abelian) Heisenberg group, which is Gaussian. It also contains the nonabelian group $SO_{4}$, which is also Gaussian. It then follows from Corollary \ref{cor:Gausssubgroups} that $\Gauss(O_{4}^{+})$ contains a classical non-abelian group and a dual of a non-commutative discrete group. For $N > 4$, simply observe that $O_{N}^{+}$ contains $O_{4}^{+}$, hence a similar containment holds for their Gaussian parts, again by Corollary \ref{cor:Gausssubgroups}.

As for $U_{N}^{+}$, it contains the dual of the free group $\mathbb{F}_{N}$, hence the free $2$-nilpotent torsion-free group, as well as the connected compact group $U_{N}$.
\end{proof}

We do not know whether $U_{N}^{+}$ is Gaussian or not. Let us notice however that the problem can be reduced to the case $N = 2$ thanks to the following observation.

\begin{lemma}
If $\QG$ is topologically generated by Gaussian quantum subgroups $(\QG_{i})_{i\in I}$, then it is Gaussian.
\end{lemma}

\begin{proof}
Let $I$ be a Hopf *-ideal on which all Gaussian functionals of $\QG$ vanish. Given a Gaussian functional $\phi$ in $\Pol(\QG_{i})$, $\phi\circ\pi_{i}$ is a Gaussian functional on $\Pol(\QG)$, hence it vanishes on $I$. In other words, $\pi_{i}(I)$ is a Hopf *-ideal annihilating all Gaussian functionals, hence $\pi_{i}(I) = 0$ by assumption. This is in turn equivalent to $I\subset \bigcap_{i\in I}\ker(\pi_{i})$, which by topological generation forces $I = 0$.
\end{proof}

Because $U_{N}^{+}$ is topologically generated by $U_{N-1}^{+}$ and $U_{N}$ for all $N\geqslant 3$ by \cite{chirvasitu20}, it follows by induction that if $U_{2}^{+}$ is Gaussian, then $U_{N}^{+}$ is Gaussian for all $N\geqslant 2$. Nevertheless, it seems difficult to understand all the Gaussian processes on $U_{2}^{+}$. One strategy to prove Gaussianity would be to find enough Gaussian quantum subgroups to topologically generate everything. For instance $U_{2}$ and the dual of the Heisenberg group $\QH_{3}$ are both Gaussian quantum subgroups of $U_{2}^{+}$, leading to the question: is $U_{2}^{+}$ topologically generated by $U_{2}$ and the dual of $\QH_{3}$? More generally one may ask: for which quotients $\Gamma$ of $\mathbb{F}_2$ is $U_{2}^{+}$ topologically generated by $U_{2}$ and the dual of $\Gamma$?




\begin{thebibliography}{CDPRR14}

\bibitem[Abe04]{abe2004hopf}
E.~Abe.
\newblock {\em {Hopf algebras}}.
\newblock Cambridge University Press, 2004.

\bibitem[ASvW88]{ASW}
L.~Accardi, M.~Sch\"{u}rmann, and W.~von Waldenfels.
\newblock Quantum independent increment processes on superalgebras.
\newblock {\em Math. Z.}, 198(4):451--477, 1988.

\bibitem[BBC07]{bb07}
Teodor {Banica}, Julien {Bichon}, and Beno\^{\i}t {Collins}.
\newblock {The hyperoctahedral quantum group}.
\newblock {\em {J. Ramanujan Math. Soc.}}, 22(4):345--384, 2007.

\bibitem[Bic04]{bichon2004free}
J.~Bichon.
\newblock Free wreath product by the quantum permutation group.
\newblock {\em Algebr. Represent. Theory}, 7(4):343--362, 2004.

\bibitem[BS09]{banica2009liberation}
T.~Banica and R.~Speicher.
\newblock {Liberation of orthogonal Lie groups}.
\newblock {\em Adv. Math.}, 222(4):1461--1501, 2009.

\bibitem[CDPR14]{connected14}
L.S. {Cirio}, A.~{D'Andrea}, C.~{Pinzari}, and S.~{Rossi}.
\newblock {Connected components of compact matrix quantum groups and finiteness
  conditions.}
\newblock {\em {J. Funct. Anal.}}, 267(9):3154--3204, 2014.

\bibitem[Chi20]{chirvasitu20}
A.~Chirvasitu.
\newblock {Topological generation results for free unitary and orthogonal
  groups.}
\newblock {\em {Internat. J. Math.}}, 31(1):2050003, 13, 2020.

\bibitem[CHK17]{CHK}
A.~Chirvasitu, S.O. Hoche, and P.~Kasprzak.
\newblock Fundamental isomorphism theorems for quantum groups.
\newblock {\em Expo. Math.}, 35(4):390--442, 2017.

\bibitem[CS19]{CaspersSkalski}
M.~Caspers and A.~Skalski.
\newblock On {$\rm C^*$}-completions of discrete quantum group rings.
\newblock {\em Bull. Lond. Math. Soc.}, 51(4):691--704, 2019.

\bibitem[Daw16]{Daws}
M.~Daws.
\newblock Categorical aspects of quantum groups: multipliers and intrinsic
  groups.
\newblock {\em Canad. J. Math.}, 68(2):309--333, 2016.

\bibitem[DFKS18]{das+al18}
B.~{Das}, U.~{Franz}, A.~Kula, and A.~{Skalski}.
\newblock {L\'evy-Khintchine decompositions for generating functionals on
  algebras associated to universal compact quantum groups.}
\newblock {\em {Infin. Dimens. Anal. Quantum Probab. Relat. Top.}}, 21(3):36,
  2018.
\newblock Id/No 1850017.

\bibitem[DK94]{dijkhuizen1994CQG}
M.S. Dijkhuizen and T.H. Koornwinder.
\newblock {CQG algebras : a direct algebraic approach to compact quantum
  groups}.
\newblock {\em Lett. Math. Phys.}, 32:315--330, 1994.

\bibitem[FKS16]{FranzKulaSkalski}
U.~Franz, A.~Kula, and A.~Skalski.
\newblock L\'{e}vy processes on quantum permutation groups.
\newblock In {\em Noncommutative analysis, operator theory and applications},
  volume 252 of {\em Oper. Theory Adv. Appl.}, pages 193--259.
  Birkh\"{a}user/Springer, [Cham], 2016.

\bibitem[GKO16]{grigorchuk2016constructions}
R.~Grigorchuk, R.~Kravchenko, and A.Y Olshanskii.
\newblock Constructions of torsion-free countable, amenable, weakly mixing
  groups.
\newblock {\em L’Enseignement Math{\'e}matique}, 61(3):321--342, 2016.

\bibitem[Hey77]{Heyer}
H.~Heyer.
\newblock {\em Probability measures on locally compact groups}.
\newblock Ergebnisse der Mathematik und ihrer Grenzgebiete, Band 94.
  Springer-Verlag, Berlin-New York, 1977.

\bibitem[KS98]{ks98}
L.~I. Korogodski and Y.~S. Soibelman.
\newblock {\em Algebras of functions on quantum groups. {P}art {I}}, volume~56
  of {\em Mathematical Surveys and Monographs}.
\newblock American Mathematical Society, Providence, RI, 1998.

\bibitem[{Las}57]{lashof1957lie}
R.~{Lashof}.
\newblock {Lie algebras of locally compact groups}.
\newblock {\em Pacific J. Math}, 7(2):1145--1162, 1957.

\bibitem[Lia04]{Liao}
Ming Liao.
\newblock {\em L\'{e}vy processes in {L}ie groups}, volume 162 of {\em
  Cambridge Tracts in Mathematics}.
\newblock Cambridge University Press, Cambridge, 2004.

\bibitem[Mag35]{magnus1935beziehungen}
W.~Magnus.
\newblock {Beziehungen zwischen Gruppen und Idealen in einem speziellen Ring}.
\newblock {\em Math. Ann.}, 111(1):259--280, 1935.

\bibitem[Mal49]{Malcev}
A.~I. Mal'cev.
\newblock Nilpotent torsion-free groups.
\newblock {\em Izvestiya Akad. Nauk. SSSR. Ser. Mat.}, 13:201--212, 1949.

\bibitem[Mey08]{meyer2008homological}
R.~Meyer.
\newblock {Homological algebra in bivariant K-theory and other triangulated
  categories. II}.
\newblock {\em Tbil. Math. J.}, 1:165--210, 2008.

\bibitem[NT12]{nt12}
S.~{Neshveyev} and L.~{Tuset}.
\newblock Quantized algebras of functions on homogeneous spaces with {P}oisson
  stabilizers.
\newblock {\em Comm. Math. Phys.}, 312(1):223--250, 2012.

\bibitem[NT13]{nt13}
S.~{Neshveyev} and L.~{Tuset}.
\newblock {\em {Compact quantum groups and their representation categories.}},
  volume~20.
\newblock Paris: Soci\'et\'e Math\'ematique de France (SMF), 2013.

\bibitem[{Pas}79]{passi1979group}
I.B.S. {Passi}.
\newblock {\em {Group rings and their augmentation ideals}}.
\newblock Number 175 in Lecture Notes in Mathematics. Springer, 1979.

\bibitem[Sat99]{Sato}
K.~Sato.
\newblock {\em L\'{e}vy processes and infinitely divisible distributions},
  volume~68 of {\em Cambridge Studies in Advanced Mathematics}.
\newblock Cambridge University Press, Cambridge, 1999.
\newblock Translated from the 1990 Japanese original, Revised by the author.

\bibitem[{Sch}90]{sch90}
Michael {Sch\"urmann}.
\newblock {Gaussian states on bialgebras}.
\newblock {Quantum probability and applications V, Proc. 4th Workshop,
  Heidelberg/FRG 1988, Lect. Notes Math. 1442, 347-367 (1990).}, 1990.

\bibitem[{Sch}93]{schurmann93}
M.~{Sch\"urmann}.
\newblock {\em {White noise on bialgebras.}}, volume 1544.
\newblock Berlin: Springer-Verlag, 1993.

\bibitem[{Sol}05]{soltan2005quantum}
P.~{Soltan}.
\newblock {Quantum Bohr compactification}.
\newblock {\em {Ill. J. Math.}}, 49(4):1245--1270, 2005.

\bibitem[SS98]{SchurmannSkeide}
M.~Sch\"{u}rmann and M.~Skeide.
\newblock Infinitesimal generators on the quantum group {${\rm SU}_q(2)$}.
\newblock {\em Infin. Dimens. Anal. Quantum Probab. Relat. Top.},
  1(4):573--598, 1998.

\bibitem[Tim08]{timmermann2008invitation}
T.~Timmermann.
\newblock {\em {An invitation to quantum groups and duality. From Hopf algebras
  to multiplicative unitaries and beyond}}.
\newblock EMS Textbooks in Mathematics. European Mathematical Society, 2008.

\bibitem[Tom07]{tom07}
R.~Tomatsu.
\newblock A characterization of right coideals of quotient type and its
  application to classification of {P}oisson boundaries.
\newblock {\em Comm. Math. Phys.}, 275(1):271--296, 2007.

\bibitem[Wan09]{wang2009simple}
Sh. Wang.
\newblock Simple compact quantum groups i.
\newblock {\em J. Funct. Anal.}, 256(10):3313--3341, 2009.

\bibitem[Wor98]{woronowicz1995compact}
S.L. Woronowicz.
\newblock {Compact quantum groups}.
\newblock {\em Sym{\'e}tries quantiques (Les Houches, 1995)}, pages 845--884,
  1998.

\end{thebibliography}
\end{document}